\documentclass[10pt,reqno]{amsart}

\usepackage[utf8]{inputenc}
\usepackage[T1]{fontenc}

\usepackage{amsmath, amssymb, amsthm}

\usepackage[foot]{amsaddr}
\usepackage{graphicx}
\usepackage{csquotes}

\usepackage{fourier}
\usepackage{enumerate}
\usepackage{microtype}
\usepackage{listings}
\usepackage{thmtools}
\usepackage{thm-restate}
\usepackage{hyperref}
\hypersetup{
  colorlinks,
  citecolor=Blue,
  linkcolor=Fuchsia,
  urlcolor=OliveGreen
}
\usepackage[capitalize, nameinlink]{cleveref}

\usepackage[
backend=biber,
isbn=false,
giveninits=true,
style=alphabetic]
{biblatex}
\usepackage[dvipsnames]{xcolor}
\usepackage[scale=0.8]{sourcecodepro}

\title{Lower Bounds for the Shadiness Constant of Finite-Dimensional Normed Spaces}
\author{Jeremias Epperlein}
\address{University of Passau\\
  Department of Computer Science and Mathematics\\
  94032 Passau, Germany}
\email{jeremias.epperlein@uni-passau.de}
\date{\today}

\newcommand{\bbR}{\mathbb{R}}
\newcommand{\bbN}{\mathbb{N}}
\newcommand{\bbQ}{\mathbb{Q}}

\newcommand{\calC}{\mathcal{C}}
\newcommand{\calD}{\mathcal{D}}

\newcommand{\calF}{\mathcal{F}}
\newcommand{\calG}{\mathcal{G}}

\newcommand{\calI}{\mathcal{I}}
\newcommand{\calJ}{\mathcal{J}}

\newcommand{\calM}{\mathcal{M}}
\newcommand{\calN}{\mathcal{N}}

\newcommand{\eps}{\varepsilon}

\newcommand{\setsep}{\;;\;}
\newcommand{\name}[1]{\textsc{#1}}

\newcommand{\norm}[1]{\lVert#1\rVert}
\newcommand{\abs}[1]{\lvert#1\rvert}
\DeclareMathOperator{\tr}{tr}
\DeclareMathOperator{\GL}{GL}
\DeclareMathOperator{\aff}{aff}
\DeclareMathOperator{\conv}{conv}
\DeclareMathOperator{\linspan}{span}

\declaretheorem[name=Theorem, numberwithin=section]{theorem}
\declaretheorem[name=Lemma, sibling=theorem]{lemma}
\declaretheorem[name=Proposition, sibling=theorem]{proposition}
\declaretheorem[name=Corollary, sibling=theorem]{corollary}
\declaretheorem[name=Example, sibling=theorem]{example}

\newtheorem{oproblem}[theorem]{Problem}
\newtheorem{suboproblem}{Problem}

\begin{document}
\begin{abstract}
  By the Hahn-Banach theorem, every normed space admits rank-one
  projections with operator norm one.  However, this is not true for
  higher rank projections.  \name{Bosznay} and \name{Garay} showed
  that for every $d \geq 3$ there exist $d$-dimensional normed spaces $X$ for which all
  projections of rank $k$, with $2 \leq k \leq d-1$, have norm larger
  than or equal to some constant $c>1$.  We call the maximal such
  constant the shadiness constant of $X$.  Although constructing such
  spaces is not difficult, few explicit estimates of their shadiness
  constants exist.  We show how optimization techniques can provide
  provable lower bounds for these shadiness constants.
  
  As an application, we construct a $3$-dimensional normed space
  whose unit ball is a polytope with $12$ vertices,
  with shadiness constant at least $1.01$. Furthermore
  we show that there is no shady norm on $\bbR^3$
  whose unit ball is a polytope with $10$ or fewer vertices, thereby
  confirming a conjecture by 
  \name{Bosznay} and \name{Garay}.
\end{abstract}

\maketitle
\section{Introduction}
\label{sec:introduction}
Imagine that you want to order a very large decorative piece of rock
for your garden. For aesthetic reasons it should have a centrally
symmetric convex polytope.  The cheapest delivery option is a drop
from a helicopter over your yard.  While the drop will ensure that the
rock sinks exactly halfway into the ground, the direction in which it
will end up cannot be controlled.  Now assume that you are also an
avid slug breeder and your little Carpathian blue slugs need shadow
throughout the day.  You therefore need your rock to be in such a shape, that
it always casts a shadow, regardless of the direction the light is
coming from. In the first part of our paper, we will construct such a
centrally symmetric convex body and prove a lower bound on the
guaranteed size of that shadow.  In the second part, we will show that
every rock satisfying this shadiness condition must have at least 12
vertices.

In more mathematical terms, we are going to construct
a three-dimensional normed space, whose unit ball is a polytope
with 12 vertices, in which every rank-two projection
has operator norm at least $1.01$. Furthermore, we will
show that in every three-dimensional normed space, whose unit ball
is a polytope with fewer than 12 vertices, there is
a rank-two projection with operator norm one.

Investigating the norms of projections is a classical topic in the theory
of finite dimensional normed spaces, see
e.g. \cite{bohnenblustConvexRegionsProjections1938,
  grunbaumProjectionConstants1960,
  kadetsFunctionalsCompactMinkovskii1971,
  koenigFiniteDimensionalProjection1983,
  konigProjectionConstantsSymmetric1999},
for a computational approach see \cite{foucartComputationMinimalProjections2016}.
It is a direct consequence of the Hahn-Banach theorem
that in every normed space every one-dimensional subspace is the range
of a norm-one projection. This is drastically wrong for higher
dimensional subspaces. \name{Bosznay} and \name{Garay} showed in
\cite{bosznayNormsProjections1986}
that for
every dimension $d \geq 3$, there is a dense set of normed spaces, in
which no projection of rank $k \in \{2,\dots,d-1\}$ has norm one.
There are even earlier examples of such normed spaces, e.g. by \name{Singer}
in \cite[Chapter II, Theorem 1.1,
p. 217]{singerBasesBanachSpaces1970} but few explicit
examples appear in the literature.
In the spirit of the introductory tale we call
such normed spaces \emph{shady}.\footnote{Apparently there is no
  established name for such spaces. \name{Singer}'s example arises
  in the context of monotone bases. \name{Bosznay} and \name{Garay}
  only introduce the notation $N_1(X)$ for the space of such norms.
  \name{Kobos} defines constants $\rho_d$ and $\rho_d^H$ which equal
  $s(d)-1$
  and $s_{d-1}(d)-1$, respectively, in our notation, but he does not
  give these constants or the corresponding spaces a name.}
The author's original interest in such norms
arose in the study of the joint spectral radius of principal
submatrices, see \cite{epperlein2025auerbach}.

In the conclusion of their
foundational article \cite{bosznayNormsProjections1986} \name{Bosznay} and \name{Garay} write:
\enquote{In the three-dimensional real case it is not hard to
  construct a centrally symmetric convex polyhedron $K$ with twelve
  vertices for which $\norm{\cdot}_K \in N_1(X)$.  On the other hand,
  it seems plausible that there are no such polyhedra with ten
  vertices. Nevertheless, we are not able to prove it.}
In our paper we give a quantitative version of their first
assertion as well as a proof of the conjecture in the second part
of the quote.
More precisely, we show:
\begin{restatable}{theorem}{tenvertices}
  \label{thm:10-vertices-geometric}
  Let $\calC \subseteq \bbR^3$ be a centrally symmetric convex polytope with at most $10$
  vertices. Then there is a rank-two projection $P$ with
  $\norm{P}_\calC = 1$, in other words, $\calC$ is not shady.
\end{restatable}

For the quantitative aspect let us introduce for a norm $\nu$ on
$\bbR^d$ the
quantity
\begin{align}
  s_k(\nu):=\inf \{\nu(P) \setsep P \in \bbR^{d \times d} \text{ is a
  rank-$k$ projection}\}.
\end{align}
Here $\nu(P)$ refers to the operator norm of $P$ with respect to $\nu$.
We call $s_k(\nu)$ the \emph{shadiness constant of $\nu$
  in dimension $k$}.
Similarly we define the (global) shadiness constant
of $\nu$ as
\begin{align}
  s(\nu):= \min_{k \in \{2,\dots,d-1\}} s_k(\nu).
\end{align}
It follows directly from the definition that
$1 \leq s(\nu) \leq s_k(\nu)$.
We thus try to understand how large these quantities can get, i.e.
how shady a norm can be.
With a slight abuse of notation we therefore additionally define
\begin{align*}
  s(d) &:= \sup \{s(\nu) \setsep \nu \text{ is a norm on } \bbR^d\} \\
  s_k(d) &:= \sup \{s_k(\nu) \setsep \nu \text{ is a norm on } \bbR^d\}
\end{align*}
The finiteness of these quantities follows
for example from the Deregowska-Lewandowska-theorem, formerly known as Grünbaum's conjecture, see
\Cref{thm:der-lew}.
While it is not too hard to prove that a three-dimensional polytope is shady,
see \Cref{prop:charact-non-shady-cscp}, estimating $s(\nu)$ or $s(d)$ is
significantly more challenging. The best bounds
have been determined by \name{Kobos} in \cite{kobosHYPERPLANESFINITEDIMENSIONALNORMED2015,
  kobosUniformEstimateRelative2018,
  kobosUniformLowerBound2023}.
In dimension three, the best current lower bound for $s_2(3)=s(3)$ is given in
\cite{kobosUniformLowerBound2023}, where a computer assisted
construction of a norm $\nu$ on $\bbR^3$ with
$s(\nu) \geq 1+932 \times 10^{-23}$ is presented.  Notice
that the best upper bound in dimension three from
\cite{kobosHYPERPLANESFINITEDIMENSIONALNORMED2015} is
$s(3) \leq \frac{4}{3}-0.0007$.
We significantly improve on the lower bound with the following theorem.
\begin{theorem}
  \label{thm:lower-bound}
  The norm on $\bbR^3$, whose unit ball is the polyhedron $\calI$ depicted in \Cref{fig:optimal-icosahedron} with vertices
  \begin{align*}
        \{&(1,a,c),
        (1,b,c),    
        (c,1,a),
        (c,1,b),
        (a,c,1), 
        (b,c,1), \\
        &-(1,a,c),
        -(1,b,c),    
        -(c,1,a),
        -(c,1,b),
        -(a,c,1), 
        -(b,c,1)
    \}
  \end{align*}
  where $a=-\frac{3}{5}, b=-\frac{1}{5}, c=\frac{1}{10}$,
  has shadiness constant at least $1.01$.
\end{theorem}

\begin{figure}
  \begin{center}
    \includegraphics{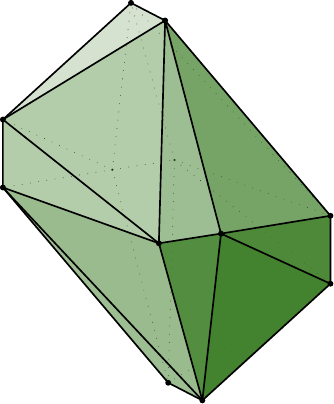}
    \caption{A centrally symmetric convex polyhedron with 12 vertices defining
      a norm on $\bbR^3$ with shadiness constant at least 1.01.}
    \label{fig:optimal-icosahedron}
  \end{center}
\end{figure}

We will describe two computer assisted proofs for this theorem, one
based on a sum-of-squares decomposition in \Cref{sec:sum-of-squares}
and another one using linear duality in \Cref{sec:eps-dense-farkas}.
In \Cref{sec:notation} we introduce our notation and discuss some preliminaries.
\Cref{sec:shadiness-constant} is devoted to the study of properties of the shadiness
constant. In \Cref{sec:calculating-norms} we formulate various optimization problems
to compute the shadiness constant of a polytopal norm. These are then
used in  \Cref{sec:sum-of-squares} and \Cref{sec:eps-dense-farkas} to
prove \Cref{thm:lower-bound}. Finally \Cref{thm:10-vertices-geometric}
is proven in \Cref{sec:non-shadiness}. 
The implementation details for our computer assisted proofs of \Cref{thm:lower-bound} are then discussed in the appendix.

\section{Notation and Preliminaries}
\label{sec:notation}
By $\bbN$ we denote the natural numbers starting at zero.
The real and rational numbers are denoted by $\bbR$ and $\bbQ$,
respectively.
We denote the positive and non-negative reals by
$\bbR_{>0}$ and $\bbR_{\geq 0}$.
Similarly, $\bbQ_{>0}$ and $\bbQ_{\geq 0}$ denote
the positive and non-negative rationals.
Given a norm $\norm{\cdot}$ on $\bbR^d$, we denote its unit ball
by $B_{\norm{\cdot}}$. To simplify notation, we use $B_2$ and
$B_\infty$
for the unit balls of, respectively, the Euclidean and the supremum norm.
The \emph{operator norm} on $\bbR^{d \times d}$ corresponding to $\norm{\cdot}$
is given by 
\begin{align*}
    \norm{A} := \sup_{\norm{x}=1} \norm{Ax}.
\end{align*}
By a projection on $\bbR^d$ we mean a linear map $P$ from $\bbR^d$
to itself satisfying $P^2=P$.

We use standard notation from convex geometry
as it can be found for example in \cite{grunbaumConvexPolytopes2003}.
For a subset $M \subseteq \bbR^d$ we denote by $\linspan(M)$
and $\aff(M)$ the \emph{linear} and \emph{affine span} of $M$,
respectively.
We denote the orthogonal complement of a vector $w \neq 0$ by
$w^\perp := \{v \in \bbR^d \setsep w^\top v = 0\}$.
A \emph{centrally symmetric convex body} is a convex subset
$\calC \subseteq \bbR^d$, which 
contains the origin in its interior and for which  $\calC=-\calC$
holds.
Such sets are precisely those that arise as the unit balls
of norms on $\bbR^d$. The norm corresponding 
to $\calC$ is given by 
\begin{align*}
    \norm{x}_\calC := \min \{ \alpha \geq 0 \setsep x \in  \alpha \calC\}.
\end{align*}
We are in particular interested in centrally symmetric convex polytopes,
i.e.  those centrally symmetric convex bodies
which are the convex hull of finitely many points.
Let $\calC$ be such a polytope. We
call an affine hyperplane $G$ a \emph{supporting
  hyperplane} of $\calC$ if $G$ intersects
the boundary of $\calC$, but not its interior.
The intersection between $G$ and $\calC$ is then
called a \emph{face}
and we call $G$ a \emph{supporting hyperplane of that face}.
The dimension of a face equals
the dimension of its affine span.
A $d-1$ dimensional face $F$ is called a \emph{facet}.
The affine span $\aff(F)$ of a facet $F$ is the unique supporting hyperplane of $F$.
Since we assume $\calC$ to be a centrally symmetric convex body, for each
supporting hyperplane $G$ of $\calC$ there is a unique vector $h_G \in \bbR^d$
such that 
\begin{align*}
  G = \{ x \in \bbR^d \setsep h_G^\top x = 1\}.
\end{align*}
$\calC$ then lies in the closed half-space $\{ x \in \bbR^d \setsep h_G^\top x \leq 1\}$.
Furthermore, if we intersect these half-spaces for all supporting
hyperplanes
of facets of $\calC$, then we get back $\calC$, i.e
\begin{align*}
  \calC = \{ x \in \bbR^d \setsep h_{\aff(F)}^\top x \leq 1, F \text{ is a
  facet of }\calC\}.
\end{align*}
This is called the \emph{H-representation} of $\calC$.
We call
$H:= \{h_{\aff(F)} \setsep F \setsep \text{ is a facet of } \calC\}$
the \emph{normals} of $\calC$.

A norm is called \emph{polytopal}, iff its unit ball is a polytope.
We will often compare polytopal norms to the Euclidean norm via the
following lemma.
\begin{lemma}
  \label{lem:compare-to-euclidean}
  Let $\calC \subseteq \bbR^d$ be a centrally symmetric convex polytope with
  vertex set $V$ and normals $H$.
  Define
  \begin{align*}
    r_\calC &:= \min\{\frac{1}{\norm{h}_2} \setsep h \in H\} \\
    R_\calC &:= \max\{\norm{v}_2 \setsep v \in V\} \\
    C_\calC &:= \frac{R_\calC}{r_\calC}
  \end{align*}
  Then $r_\calC$ is the radius of the largest Euclidean ball around
  the origin contained in
  $\calC$
  and $R_\calC$ is the radius of the smallest Euclidean ball around the
  origin containing $\calC$.
  This implies
  \[r_\calC \norm{x}_\calC \leq \norm{x}_2 \leq R_\calC \norm{x}_\calC
    \text{ for all }x
    \in \bbR^d\]
  and
  \[\frac{1}{C_\calC} \norm{A}_2 \leq \norm{A}_\calC \leq C_\calC\norm{A}_2 \text{ for all } A \in \bbR^{d \times d}.\]
\end{lemma}
\begin{proof}
  This follows from elementary Euclidean geometry.
\end{proof}

For $\delta >1$
we call a centrally symmetric convex body $\calC \subseteq \bbR^d$ and
the corresponding norm $\norm{\cdot}_\calC$ $\delta$-\emph{shady in
  dimension $k$} or
 \emph{simply shady} in dimension $k$ if 
$\norm{P}_\calC \geq \delta$
for every rank-$k$ projection $P$.
With respect to the shadiness constant $s_k$ introduced in the introduction 
this is equivalent to $s_k(\nu) \geq \delta$.
Similarly we call $\calC$ and $\norm{\cdot}_\calC$
(globally) \emph{$\delta$-shady} or simply \emph{shady}
if $s(\nu)\geq \delta$.
Hence $\calC$ is globally shady if it is shady in all dimensions
from $2$ to $d-1$.
We are in particular interested in shadiness in dimension $d-1$.
Of course, for $d=3$, this is equivalent to global shadiness.

The following proposition and corollary give a simple characterization
of $d$-dimensional centrally symmetric convex polytopes which are
shady in dimension $d-1$.
A similar result appears in \cite{singerBasesBanachSpaces1970}
as Lemma 1.2.
\begin{proposition}
  \label{prop:charact-non-shady-cscp}
  Let $\calC \subseteq \bbR^{d}$ be a centrally symmetric convex polytope.
  Then $\calC$ is not shady in dimension $d-1$ iff
  there are $d-1$ dimensional subspaces
  $U$ and $W$ such that
  the relative interior of every face of $\calC$, which is intersected by $U$,
  admits a supporting hyperplane with normal in $W$.
\end{proposition}
\begin{proof}
  Assume first that there are subspaces $U$ and $W$
  as described above.
  Since $W$ is $(d-1)$-dimensional there is a vector $w \in \bbR^d
  \setminus \{0\}$ perpendicular to $W$.
  We first show that $w$ is not contained in $U$. Assume otherwise
  that $w \in U$.
  The boundary of $\calC$ is the disjoint union of the relative
  interior of its faces. Hence there is
  a unique pair of opposite faces of $\calC$ whose relative interior
  is intersected by $\linspan(\{w\})$.
  Let $F$ be one of these faces. By assumption there is a
  supporting hyperplane $H$ of $F$ with normal $h \in W$.
  Now both $H$ and $w$ are perpendicular to $h \in W$,
  hence $\linspan(\{w\})$ must be parallel to $H$. Since $\linspan(\{w\})$ intersects
  $F \subseteq H$, this implies that $\linspan(\{w\})$ must
  actually be contained in $H$. But this contradicts
  the fact that the origin is contained in $\linspan(\{w\})$ but not in $H$.

  Now, since $w \not\in U$, there is a unique projection $P$ with image $U$ and kernel spanned by
  $w$. Let $V$ be the set of vertices of $\calC$.
  In order to show that $\norm{P}_{\calC}=1$, we have
  to show that $P(V) \subseteq \calC$,
  since then also $P(\calC)=P(\conv(V)) \subseteq \calC$.
  Assume otherwise that there exists a vertex $v \in V$
  with $Pv \not \in \calC$.
  There is $\alpha \in \bbR$ such that $v = Pv + \alpha w$.
  Let $F$ be the unique face of $\calC$ whose relative interior is
  intersected by the ray $\{\beta Pv \setsep \beta \geq 0\}
  \subseteq U$.
  Since $U$ intersects the relative interior of $F$,
  there is a vector $h \in W$ perpendicular to $F$.
  We may normalize $h$ such that $\{x \in \bbR^d \setsep h^\top x = 1\}$
  is a supporting hyperplane of $F$.
  Since $Pv \not \in \calC$, the open line segment from the origin to $Pv$ intersects
  $F$ and we have
  $h^\top Pv >1$.
  On the other hand,
  $h^\top Pv = h^\top v - \alpha h^\top w = h^\top v \leq 1$,
  contradiction.
  We thus have shown that $\norm{P}_\calC=1$ and hence that $\calC$ is
  not shady in dimension $d-1$.

  On the other hand assume that $\calC$ is not shady
  and let $P$ be a projection of rank $d-1$ and
  norm one. Let $w$ be a vector spanning the kernel of $P$
  and let $W=w^\perp$.
  Let $U$ be the range of $P$.
  Both $U$ and $W$ are $(d-1)$-dimensional linear subspaces of
  $\bbR^d$.
  Let $F$ be a face of $\calC$ whose interior intersects $U$.
  Since $P$ has norm one, the interior of $\calC$
  must be projected to the interior of $\calC$ under $P$.
  Hence the convex subset $M:=F+ \linspan(\{w\})$, which is projected to $F \cap U$ and thus
  to the boundary of $\calC$, must be disjoint from the interior of
  $\calC$.
  By the hyperplane separation
  theorem there is a hyperplane $H$
  separating the interior of $\calC$ from $M$.
  $H$ must be parallel to $\linspan(\{w\})$
  and it must contain $F$. Therefore any vector perpendicular to $H$
  is also perpendicular to $F$ and thus contained in  $w^\perp = W$. 
\end{proof}  

Based on this result we can give a simple sufficient criterion for
a centrally symmetric convex polytope to be shady in dimension $d-1$.

\begin{corollary}
  \label{cor:shadiness-in-dim-d}
  Let $\calC \subseteq \bbR^d$ be a centrally symmetric convex polytope.
  Let $H$ be a set of normals of supporting hyperplanes
  of $\calC$, one for each pair of opposite facets.
  If the vectors in $H$ are in general position, i.e.
  every $d$ tuple in $H$ is linearly independent,
  and every $(d-1)$-dimensional linear subspace of $\;\bbR^d$ intersects
  the relative interior of at least $d$ pairs of opposite facets of $\calC$, then
  $\calC$ is shady
  in dimension $d-1$.
\end{corollary}
\begin{proof}
  Assume that $\calC$ is not shady in dimension $d-1$. Let $U$
  and $W$ be the linear subspaces
  from \Cref{prop:charact-non-shady-cscp}.
  By assumption $U$ intersects the relative interior of $d$ pairs of
  opposite facets.
  For each of these pairs of facets there is a unique normal in $H$ and
  hence these normals are linearly independent by assumption.
  But they are also all contained in the $d-1$ dimensional subspace $W$, contradiction.
\end{proof}

\begin{example}
  \label{ex:simple-shadiness-test}
  In order to apply the corollary to the polyhedron $\calI$ from
  \Cref{thm:lower-bound},
  let $H$, for example, be the set of normals to facets to $\calI$
  whose first coordinate is positive.
  A straightforward but somewhat lengthy check shows that $H$ is in general position.
  Similarly one can show that every two-dimensional linear subspace
  intersects the relative interior of at least two pairs of opposite
  facets of $\calI$, see \Cref{sec:impl-general}.
  This already shows that $\calI$ is shady, but
  does not give any lower bound for the shadiness constant $s_2(\calI)$.
\end{example}

In \Cref{sec:sum-of-squares} we will use some techniques from real
algebraic geometry, where we will use the following notation, most of
which can be found e.g. in \cite{theobaldRealAlgebraicGeometry2024}.  Let us
consider variables $x = (x_1,\dots,x_n)$.  Let
$\bbR[x] = \bbR[x_1,\dots,x_n]$ be the ring of real polynomials in the commuting
scalar variables $x_1,\dots,x_n$.  For a multi-index $\alpha \in \bbN^n$ we
denote by $x^\alpha$ the monomial
$x_1^{\alpha_1} \cdots x_n^{\alpha_n}$.
Let $\abs{\alpha}=\sum_{i=1}^n \alpha_i$ be the \emph{weight} of
$\alpha$.
We denote the set of multi-indices of
weight at most $r$ by
\begin{align*}
  \bbN^n_{\leq r} = \{ \alpha \in \bbN^n \setsep \abs{\alpha}\leq r\}.
\end{align*}
It is well-known that $\abs{\bbN_{\leq r}^n} = \binom{n+r}{n}$.
By
\begin{align*}
    \bbR[x_1,\dots,x_n]^2:=\{\sum_{i=1}^k s_i^2 \setsep k \in \bbN, s_i \in \bbR[x_1,\dots,x_n] \text{ for } i \in \{1,\dots,k\}\}
\end{align*}
we denote the cone of \emph{sum-of-square} polynomials.
We will also need rational sum-of-square polynomials. Here we allow
additional positive rational coefficients\footnote{
  In the real case, we can incorporate these coefficients into the
  squared polynomials. This is no longer possible in the
  rational case.} and define
\begin{align*}
  \bbQ[x_1,\dots,x_n]^2 := \{\sum_{i=1}^k \gamma_i s_i^2 \setsep k \in
  \bbN,\, s_i \in \bbQ[x_1,\dots,x_n],\, \gamma_i \in \bbQ_{>0} \text{ for } i \in \{1,\dots,k\}\}.
\end{align*}

For a family of polynomials $\calF=\{f_1,\dots,f_\ell\} \subseteq \bbR[x]$
let
\begin{align*}
    I(\calF):= \{ \sum_{i=1}^\ell p_i f_i \setsep p_i
  \in \bbR[x] \text{ for } i \in \{1,\dots,\ell\}\}
\end{align*}
be the \emph{ideal} generated by $\calF$ and let 
\begin{align*}
    M(\calF):= \{q_0 + \sum_{i=1}^\ell q_i f_i \setsep q_i \in
  \bbR[x]^2 \text{ for }i \in \{0,\dots,\ell\}\}
\end{align*}
be the \emph{quadratic module} generated by $\calF$.
We also define rational versions of these structures:
\begin{align*}
    I_\bbQ(\calF)&:= \{ \sum_{i=1}^\ell p_i f_i \setsep p_i \in
                   \bbQ[x] \text{ for } i \in \{1,\dots,\ell\}\} \\
  M_\bbQ(\calF)&:= \{q_0 + \sum_{i=1}^\ell q_i f_i \setsep q_i \in
                 \bbQ[x]^2 \text{ for } i \in \{0,\dots,\ell\}\}
\end{align*}

A quadratic module $M$ is \emph{Archimedean} if
there is $N \in \bbN$ (or equivalently $N \in \bbR_{>0}$) such
that\[N-\sum_{i=1}^n x_i^2 \in M.\]

A set of the form
\[\{x \in \bbR^n \setsep f(x) \geq 0 \text{ for all } f \in \calF\}\]
is called a \emph{basic semialgebraic set}.

We finish this section with two short and easy results which we will
need in \Cref{sec:sum-of-squares}.
Both are certainly well-known but we could not find them in this
precise form in the literature.
\begin{lemma}
\label{lem:decompose-1-x}
Let $x=(x_1,\dots,x_n)$ be variables and let $\alpha \in \bbN^n$.
Then there are polynomials $\tilde{p}_1,\dots,\tilde{p}_n$
in $\bbR[x]$ with non-negative integer coefficients
such that 
\begin{align*}
    1-x^\alpha = \sum_{i=1}^n (1-x_i) \tilde{p}_i(x_1,\dots,x_n). 
\end{align*}
\end{lemma}
\begin{proof}
    We prove this by induction on the weight of $\alpha$.
    If the weight of $\alpha$ is one or zero, the theorem follows immediately.
    Otherwise assume without loss of generality that $\alpha_1 \neq 0$.
    Set $\beta:=\alpha - (1,0,\dots,0)$.
    Then
    \begin{align*}
        1-x^\alpha = (1-x_1)+x_1 \cdot (1-x^\beta).
    \end{align*}
    Applying the induction hypothesis to $(1-x^\beta)$  we obtain our result.
\end{proof}

\begin{proposition}
  \label{thm:offset}
  Let $x=(x_1,\dots,x_n)$ be variables, let $\Omega \in \bbQ_{>0}$ and
  let $\alpha \in \bbN^n$.
  Set $\Delta := \Omega^{2\abs{\alpha}}$.
    There are rational sum-of-square polynomials $p_1,\dots,p_n \in \bbQ[x]^2$ such that
    \[\Delta-x^{2\alpha} = \sum_{i=1}^n p_i \cdot (\Omega^2 - x_i^2).\]
\end{proposition}
\begin{proof}
  Consider the polynomials $\tilde{p}_1, \dots,\tilde{p}_n$
  from \Cref{lem:decompose-1-x} applied to $\alpha$.
    Replace the variables by $(x_1^2/\Omega^2,\dots,x_n^2/\Omega^2)$.
    After multiplying both sides by $\Omega^{2\abs{\alpha}}$
    we get
    \begin{align*}
        \Omega^{2\abs{\alpha}} - x^{2\alpha} &= \sum_{i=1}^n \Omega^{2(\abs{\alpha}-1)} \tilde{p}_i(\frac{x_1^2}{\Omega^2},\dots,\frac{x_n^2}{\Omega^2})(\Omega^2-x_i^2) 
    \end{align*}
    We can now finish the proof by setting
    \begin{align*}
    p_i(x_1,\dots,x_n)&:=\Omega^{2(\abs{\alpha}-1)}
            \tilde{p}_i(\frac{x_1^2}{\Omega^2},\dots,\frac{x_n^2}{\Omega^2}). \qedhere
    \end{align*}
  \end{proof}
  
\section{The Shadiness Constant}
\label{sec:shadiness-constant}
In this section we collect some useful facts on the shadiness
constants
\[s_k(\nu)=\inf \{\nu(P) \setsep P \in \bbR^{d \times d} \text{ is a
    rank-$k$ projection}\}\] defined in the introduction.
To simplify notation we will also write
$s_k(\calC)=s_k(\norm{\cdot}_\calC$
whenever $\calC$ is a centrally symmetric convex body.
The \emph{relative projection constant} is a related
quantity which received considerable attention in
Banach space geometry, see for example
\cite{grunbaumProjectionConstants1960},
\cite{koenigFiniteDimensionalProjection1983} and
\cite{kobosUniformEstimateRelative2018}.
For a Banach space $X$ and a subspace $Y$
it is defined as
\begin{align*}
  \lambda(Y,X) = \inf\{\norm{P} \setsep P \text{ is a projection from $X$ to $Y$}\}.
\end{align*}
The shadiness constant in dimension $k$ is the minimum
of the relative projection constants over all $k$-dimensional
subspaces, i.e.
\begin{align*}
  s_k(\calC) = \min \{\lambda(Y, (\bbR^d, \norm{.}_\calC)) \setsep Y
  \leq \bbR^d, \dim Y= k\}. 
\end{align*}

A useful simple result about the shadiness constant is the fact that the shadiness constant does
not change under similarity transformations.
\begin{lemma}
    \label{lem:transformation-shadiness}
    Let $T \in \GL_d(\bbR)$ and let $\calC \subseteq \bbR^d$ be  a
    centrally symmetric convex body.
    Then $\norm{x}_{T\calC} = \norm{T^{-1}x}_{\calC}$ for $x \in \bbR^d$
    and $\norm{A}_{T\calC} = \norm{T^{-1}AT}_{\calC}$ for $A \in
    \bbR^{d \times d}$.
    For all $k \in \{2,\dots,d-1\}$ we have $s_k(\calC)=s_k(T\calC)$.
\end{lemma}
\begin{proof}
    Notice that 
    $\norm{x}_{T\calC} = \min \{\alpha \geq 0 \setsep x \in \alpha T\calC\}
    = \min\{\alpha \geq 0 \setsep T^{-1}x \in \alpha\calC\} = \norm{T^{-1}x}_\calC$.
    Furthermore, $\norm{A}_{T\calC} = \sup_{x \neq 0} \frac{\norm{T^{-1}ATT^{-1}x}_\calC}
    {\norm{T^{-1}x}_\calC} = \norm{T^{-1}AT}_\calC$.
    Now assume that $\calC$ is $\delta$-shady in dimension $k$.
    Let $P$ be a rank-$k$ projection. We have to show that
    $\norm{P}_{T\calC} \geq \delta$.
    This follows directly from $\norm{P}_{T\calC} = \norm{T^{-1}PT}_{\calC}$
    and the fact that $T^{-1}PT$ is also a rank-$k$ projection.
  \end{proof}

  By definition $s_k$ is lower bounded by one. A good upper bound can
  be obtained from the recent proof of the Grünbaum conjecture by
  \name{Deregowska} and \name{Lewandowska}, 
  \cite{deregowskaSimpleProofGrunbaum2023}. 
  \begin{theorem}[Deregowska, Lewandowska 2023]
    \label{thm:der-lew}
    For a norm $\nu$ on $\bbR^d$ we have 
    \[s_k(\nu) \leq \frac{2}{k+1}(1+\frac{k-1}{2}\sqrt{k+2}) \leq
      \sqrt{k}, \quad k \in \{2,\dots,d-1\}.\]
  \end{theorem}
  
  Next we show that the shadiness constant depends continuously on
  the norm. To make sense of this, we have to equip the space of norms
  with a topology. A common choice is to use the Banach-Mazur
  distance, see e.g. \cite{tomczak-jaegermannBanachMazurDistancesFinitedimensional1989},
  which is defined for two norms $\nu_1,\nu_2$ on $\bbR^d$ as
   \begin{align*}
     d(\nu_1,\nu_2) := \min_T \norm{T}_{\nu_1,\nu_2} \norm{T^{-1}}_{\nu_2,\nu_1},
  \end{align*}
   where the minimum is taken over all $T \in \GL_d(\bbR)$.
   Here $\norm{T}_{\nu_1,\nu_2}$ denotes
   the operator norm of $T$ considered as a map from $(\bbR^d,\nu_1)$
   to $(\bbR^d,\nu_2)$, i.e.
   $\norm{T}_{\nu_1,\nu_2} = \sup_{x \neq 0}
   \frac{\nu_2(Tx)}{\nu_1(x)}$.
   The map $\delta = \log d$ is almost a metric, but it is not definite, since linearly isometric spaces have
   distance zero. If we quotient by linear isometries, we obtain
   a compact metric space called the \emph{Banach-Mazur compactum} or
   \emph{Minkowski compactum}.
  \begin{theorem}
    Let $k \in \{2,\dots,d-1\}$.
    Then $\nu \mapsto \log s_k(\nu)$ is a $1$-Lipschitz map from the Banach-Mazur compactum
    to $\bbR$. In particular, $s_k$ depends continuously on the norm.
    Furthermore, there is a norm $\tilde{\nu}$ on $\bbR^d$ which
    minimizes $s_k$ over all norms on $\bbR^d$.
  \end{theorem}
  \begin{proof}
    Since $s_k$ is invariant under similarity transformations,
    $\nu \mapsto \log s_k(\nu)$ induces
    a well defined map on the Banach-Mazur compactum.
    
    Let $\nu_1, \nu_2$ be two norms on $\bbR^d$.
    By the definition of the Banach-Mazur distance, there is
    a linear map $T: (\bbR^d,\nu_1) \to (\bbR^d,\nu_2)$
    such that
    $\delta(\nu_1,\nu_2) = \log \norm{T}_{\nu_1,\nu_2}  \norm{T^{-1}}_{\nu_2,\nu_1} $.
    There is a rank-$k$ projection $P$ on $\bbR^d$ such that
    $s_k(\nu_1)=\nu_1(P)$.
    Now $Q:=T P T^{-1}$ is a rank-$k$ projection
    on $\bbR^d$ satisfying
    \[\nu_2(Q) \leq \norm{T}_{\nu_1,\nu_2}  \norm{T^{-1}}_{\nu_2,\nu_1}  \nu_1(P).\]
    Hence $s_k(\nu_2) \leq e^{\delta(\nu_1,\nu_2)} s_k(\nu_1)$.
    By symmetry we also get
    $s_k(\nu_1) \leq e^{\delta(\nu_1,\nu_2)} s_k(\nu_2)$.
    Taking the logarithm we obtain
    \begin{align*}
      \abs{\log s_k(\nu_1) - \log s_k(\nu_2)} \leq
      \delta(\nu_1,\nu_2).
    \end{align*}
    The final statement in the theorem then follows directly from
    the compactness of the Banach-Mazur compactum.
  \end{proof}

  The last result of this section now shows that in order to
  find bounds for the shadiness constants
  it is enough to consider polytopal norms. This was mentioned in
  \cite{epperlein2025auerbach}.
  \begin{corollary}
    Let $\Omega_d$ be the set of norms on $\bbR^d$
    and let $\Omega'_d$ be the subset of norms
    whose unit ball is a polytope.
    Then
    \[\max_{\nu \in \Omega_d} s_k(\nu) =
    \sup_{\nu \in \Omega_d'} s_k(\nu).\]
\end{corollary}
\begin{proof}
  By taking a sufficiently dense subset of the boundary of the unit ball, we can
  approximate every norm by a polytopal norm with respect to the
  Banach-Mazur distance. Together with the continuity of $s_k$
  this proves the result.
\end{proof}

\section{Operator Norms of Projections with Respect to Polytopal Norms}
\label{sec:calculating-norms}
From now in, let $\calC$ be a centrally symmetric convex polytope.
Let $V$ be the vertices of $\calC$, so $\calC$ is
the convex hull of $V$. 
For each pair of opposite vertices pick one vertex
and collect these vertices in a set $V'$,
so that $V=V' \cup -V'$.
Let $H$ be the set of normals of $\calC$. As discussed in \Cref{sec:notation},
\begin{align*}
  \calC = \{x \in \bbR^d \setsep h^\top x \leq 1 \text{ for all } h \in H\}.
\end{align*}
Set \[\calM_{k,\calC} := \{(\alpha,P) \in \bbR \times \bbR^{d \times d}
\setsep P \text{ is a rank-$k$ projection},\,\norm{P}_\calC \leq \alpha\}.\]  Using the function
\[f: \calM_{k,\calC} \to \bbR,\; (\alpha,P) \mapsto \alpha,\] we can
express the shadiness constant of $\calC$ in dimension $k$ as
\begin{align*}
  s_k(\calC)=\min_{x \in \calM_{k,\calC}} f(x).
\end{align*}
The advantage of this formulation is the fact that $f$ is a linear
function
and $\calM_{k,\calC}$ is a basic semialgebraic set, as we show now.
We start with the following explicit characterization of
$\calM_{k,\calC}$,
which uses the fact that the eigenvalues of a projection
are all either $0$ or $1$:
\begin{align*}
  \calM_{k,\calC} = \{(\alpha,P) \in \bbR \times \bbR^{d \times d} \setsep
  P^2=P,\; \tr P = k,\; P \calC \subseteq \alpha \calC \}.
\end{align*}
Now $P(\calC) \subseteq \alpha \calC$ is equivalent
to $\frac{1}{\alpha} P(V) \subseteq \calC$, which  is
equivalent
to $\frac{1}{\alpha} P(V') \subseteq \calC$ by the symmetry of $\calC$.
This in turn is equivalent
to $\frac{1}{\alpha} h^\top Pv \leq 1$ for all $h \in H, v \in V'$.
Therefore we can write $\calM_{k,\calC}$ as
\begin{align*}
  \calM_{k,\calC} = \{(\alpha,P) \in \bbR \times \bbR^{d \times d} \setsep
  P^2=P,\; \tr P = k,\; h^\top Pv \leq \alpha \text{ for all $h \in H, v \in V'$}\}.
\end{align*}
Therefore $\calM_{k,\calC}$ is characterized by $\frac{1}{2}\abs{V}\abs{H}$ linear
inequalities, one linear equation and $d^2$ quadratic equations
in the $d^2+1$ variables $P_{1,1},P_{1,2},\dots,P_{d,d}$ and $\alpha$.

All in all, we can characterize $s_k(\calC)$
as the solution to the following optimization problem.
\begin{oproblem}
  \label{op:P-formulation}
\begin{align*}
    \text{Minimize:} \quad & \alpha \\
  \text{Subject to:} \quad & P^2=P, \\
  &\tr P = k, \\
  &h^\top P v \leq \alpha,\quad v \in V', h \in H
\end{align*}
with variables $P \in \bbR^{d \times d}, \quad  \alpha \in \bbR$.
\end{oproblem}

This formulation can readily be used in a solver like SCIP,
which can find global optima of quadratic optimization
problems like ours, even in the non-convex case,
using branch-and-bound techniques. For implementation details see \Cref{sec:impl-general}.
As an approximation of the shadiness constant of the polytope
$\calI$ from \Cref{thm:lower-bound} we obtain
$1.0127$.

A rational approximation
of the optimal projection is given by
\begin{align*}
\begingroup
\renewcommand*{\arraystretch}{1.3}
  P = \begin{pmatrix}
    \frac{82602121}{79729122} & \frac{54836807}{79729122} & \frac{-722323}{13288187} \\
\frac{-4217717}{79729122} & \frac{-774259}{79729122} & \frac{1060409}{13288187} \\
\frac{695635}{39864561} & \frac{13277555}{39864561} &
                                                      \frac{12938397}{13288187}
  \end{pmatrix}.
  \endgroup
\end{align*}
Notice that this approximation is still a rank-$2$ projection.
Calculating the operator norm of $P$ exactly\footnote{As the
  calculations above show, 
  $\norm{P}_{\calC} = \max \{h^\top P v \setsep v \in V', h \in H\}.$}, we obtain
\begin{align*}
  s_2(\calI) \leq \norm{P}_\calI = \frac{14386149522}{14205071903} \approx 1.0127.
\end{align*}
Remember, however, that we are interested in provable lower
bounds for $s_2(\calI)$.

The above formulation involves $d^2$ variables.
For the special case of rank-$(d-1)$ projections,
we can significantly reduce the number of variables
in our optimization problem as follows.
Observe that a rank-$(d-1)$ projection
is uniquely determined by a non-zero vector $u$ spanning its kernel and a
non-zero vector $w$ perpendicular to its image.
The projection $P$ is then given by 
\begin{align*}
  Px &= x - \frac{w^\top x}{w^\top u}u.
\end{align*}
To see this let $x \in \bbR^d$ be an arbitrary vector.
Since $0=P^2x-Px=P(Px-x)$, there must be $\gamma \in \bbR$
such that $Px-x=\gamma u$.
Additionally $w^\top Px = 0$, so $w^\top x = -\gamma w^\top u$.
Finally we get
$Px = x- \frac{w^\top x}{w^\top u} u$.
This is defined only for $w^\top u \neq 0$, but
this is necessary anyway, since the kernel
and image of a projection intersect only in the origin.
Multiplying $w$ or $u$ by a non-zero scalar
does not change the projection. Hence we
may additionally assume that $w^\top u >0$.

Now consider $h \in H$ and $v \in V'$.
Our inequality $h^\top Pv \leq \alpha$ translates
to
\begin{align*}
  h^\top (v - \frac{w^\top v}{w^\top u} u) \leq \alpha.
\end{align*}
Multiplication by $w^\top u$ gives
\begin{align*}
  h^\top v w^\top u -w^\top v h^\top u \leq \alpha w^\top u. 
\end{align*}

Therefore we can also characterize $s_{d-1}(\calC)$
as the solution to the following optimization problem.
\addtocounter{theorem}{1}
\begin{suboproblem}
  \label{op:kernel-formulation}
\begin{align*}
    \text{Minimize:} \quad & \alpha \\
    \text{Subject to:} \quad & w^\top u > 0, \\
                           & h^\top v w^\top u-w^\top v h^\top u -
                             \alpha w^\top u\leq  0,\quad
                             v \in V', h \in H
\end{align*}
with variables $u,w \in \bbR^d$, $\alpha \in \bbR$.
\end{suboproblem}

The domain of this problem is not compact. 
We can change that by first observing that we can replace
$w^\top u>0$ by $w^\top u \geq 0, w \neq 0, u\neq 0$.
Clearly the first condition implies the second.
Now assume there would be $w, u \in \bbR^d \setminus \{0\}$
with $w^\top u =0$ satisfying the second condition 
in \Cref{op:kernel-formulation}.
This would immediately imply $w^\top v h^\top u \geq 0$
for all $v \in V', h \in H$. Since $\calC$ is centrally symmetric, we have $H=-H$,
so we would necessarily also have 
$w^\top v h^\top u=0$ for all $v \in V', h \in H$.
Since both $V'$ and $H$ span $\bbR^d$ as a vector space, and $w \neq 0, u \neq 0$
we can find $v \in V'$ and $h \in H$ such that $w^\top v h^\top u \neq 0$,
contradiction.

Finally observe that that all conditions of \Cref{op:kernel-formulation}
are invariant under scaling of $u$ and $w$ by positive scalars,
hence we may assume $u,w \in [-1,1]^d$
and $u_i,w_j \in \{-1,1\}$ for some indices
$i,j \in \{1,\dots,d\}$.
Additionally
$(u,w,\alpha)$ is feasible iff $(-u,-w,\alpha)$ is feasible,
so we may even assume $w_j=1$.
All in all this allows us 
to rewrite our optimization problem in the following form 
with a compact domain:
\begin{suboproblem}
  \label{op:kernel-formulation-compact}
\begin{align*}
    \text{Minimize:} \quad & \alpha \\
    \text{Subject to:} \quad & w^\top u \geq 0, \\
                               &u, w \in [-1,1]^d, \\
                               &\exists i \in \{1,\dots,d\}: u_i \in \{-1,1\},\\
                               &\exists j\in \{1,\dots,d\}: w_j= 1,\\
                               & h^\top v w^\top u-w^\top v h^\top u -
                             \alpha w^\top u\leq  0,\quad
                             v \in V', h \in H
\end{align*}
with variables $u,w \in \bbR^d, \alpha \in \bbR$.
\end{suboproblem}
\addtocounter{theorem}{1}

\section{A Proof Certificate Based on Sums of Squares}
\label{sec:sum-of-squares}

In this section we give a first computer assisted proof of
\Cref{thm:lower-bound} based on \Cref{op:P-formulation}. The core of the proof
will consist of the construction of a decomposition
of a negative number as a rational weighted sum-of-squares. We will call
this decomposition the \enquote{proof certificate}.
Checking its validity is simply done
by expanding this decomposition and
checking that it evaluates to a negative number.
The main part of this section is devoted to the
construction of this decomposition, but the only part really necessary
for the proof is the resulting decomposition itself.

In order to show that $\alpha^*$ is a lower bound
for \Cref{op:P-formulation} or \Cref{op:kernel-formulation-compact},
we have to show that for $\alpha < \alpha^*$ the
constraints become infeasible.
So let us consider the set
\begin{align*}
   \calM_{k,\calC}(\alpha^*) := \{P \in \bbR^{d \times d} \setsep
  P^2=P,\, \tr P = k,\, h^\top Pv \leq \alpha^* \text{ for all $v \in V',h
  \in H$}\}
\end{align*}
of rank-$k$ projections
whose $\calC$-norm is at most $\alpha^*$.
Our goal is to show that $\calM_{k,\calC}(\alpha^*)$
is empty.
Notice that we can easily bound the size of
the entries of $P \in \calM_{k,\calC}(\alpha^*)$.
Let $C_\calC$ be the constant from \Cref{lem:compare-to-euclidean},
so
\begin{align*}
  \norm{A}_2 \leq C_\calC \norm{A}_\calC 
\end{align*}
for every
$A \in \bbR^{d \times d}$.
If $\norm{P}_\calC \leq \alpha^*$, then for
each standard basis vector $e_i$ we have
\[Pe_i \in PB_2 \subseteq \norm{P}_2 B_2  \subseteq  \norm{P}_2 B_\infty\]
and hence
\begin{align}
  \abs{P_{ij}} \leq \norm{P}_2 \leq C_\calC \norm{P}_\calC \leq
  C_\calC \alpha^* \quad \text{for } i,j \in \{1,\dots,d\}.\label{eq:bound-projection-entries}
\end{align}
We can use the trace condition to eliminate one entry of $P$
by setting $P_{d,d}=k-\sum_{i=1}^{d-1} P_{i,i}$.

Writing $x=(x_1,\dots,x_n):=(P_{1,1},P_{1,2},\dots,P_{d,d-1})$ for our variables
we have to show that
\begin{align*}
  \calM = \{ x \in \bbR^{d^2-1} \setsep\quad f_i(x)=0,\; i=1,\dots,s, \quad
  g_j(x) \geq 0,\; j = 1,\dots,t \} =\emptyset,
\end{align*}
where 
the polynomials $f_i$ correspond to the $s:=d^2$-entries of
$P^2-P$
and the polynomials $g_j$ correspond
to the $\abs{V'}\abs{H}$ inequalities of the form $\alpha^* -
h^\top P v \geq 0$
together with the $d^2-1$ inequalities of the form
$(C_\calC\alpha^*)^2 - P_{i,j}^2  \geq 0$
from \eqref{eq:bound-projection-entries}.
In total this yields \[t:=\abs{V'}\abs{H}+d^2-1\]
inequalities.

So on an abstract level we
want to show that a compact basic semialgebraic set $\calM$ is empty. We can do
this by showing that the constant function $-1$ is positive over
$\calM$.
Actually, any other constant negative function works as well and
we will use this in our concrete implementation in order to
improve the size of our proof certificate.
At first, this sounds like a nonsensical rephrasing of the problem.
It is the following algebraic reformulation that actually makes it
useful for us:
Find polynomials $p_1,\dots,p_{s} \in \bbR[x]$ and
$q_0, q_1,\dots,q_t \in \bbR[x]^2$ such that
\begin{align*}
  -1 = q_0 + \sum_{j=1}^t q_j g_j + \sum_{i=1}^s p_i f_i.
\end{align*}
We call this a \emph{weighted sum-of-squares decomposition} of $-1$.
On the set $\calM$ the right hand side is clearly non-negative,
whereas the left hand side is negative, hence $M$ must be empty.

On the other hand, if $\calM$ is indeed empty, the existence of such polynomials $p_i$ and $q_j$
is guaranteed by Putinar's Positivstellensatz\footnote{Usually this
    theorem
    is stated without the polynomials $\calF$ and the corresponding
    equations. Our version is easily seen to be equivalent to this
    more standard form by noting that $I(\calF) = \calM(\calF \cup
    -\calF)$, which reduces to $p=\frac{1}{4}((p+1)^2-(p-1)^2)$.} \cite{putinarPositivePolynomialsCompact1993}:

\begin{theorem}[Putinar's Positivstellensatz]
  Consider a polynomial $h$ and families of polynomials $\calF = \{f_1,\dots,f_s\}$, $\calG = \{g_1,\dots,g_t\}$ for which the quadratic module $M(\calG) + I(\calF)$ is Archimedean.
    If $h(y)>0$ for all
    \[y \in \{x \in \bbR^n \setsep f_i(x) = 0, i=1,\dots,s, \quad
      g_j(x) \geq 0,j =1,\dots,t\}\] then $h \in  M(\calG) + I(\calF)$.
    
    In particular, $\{x \in \bbR^n \setsep f_i(x) = 0, i=1,\dots,s, \quad
      g_j(x) \geq 0,j =1,\dots,t\} = \emptyset$, iff \[-1 \in  M(\calG)+ I(\calF).\]
\end{theorem}

While this theorem guarantees the existence of the necessary polynomials,
it doesn't tell us how to find those.
The sum-of-squares method owes its usefulness to the
fact that we can find these polynomials by semidefinite programming.
This approach has a long history based on work by many people, see
e.g. \cite{marshallPositivePolynomialsSums2008,
  blekhermanSemidefiniteOptimizationConvex2012,
  theobaldRealAlgebraicGeometry2024,henrionMomentSOSHierarchyLectures2021}.
It found many applications in diverse fields
of applied mathematics and engineering discussed in the works above.
For an application in pure mathematics see e.g. \cite{kalubaPropertyPropertyT2021}
and for an application specifically in convex geometry see \cite{kellnerSumSquaresCertificates2016}.

In our case the procedure works as follows.
Let $p$ be a real polynomial.
Notice that $p \in \bbR[x]^2$ is equivalent 
to the existence of $r \in \bbN$ and a positive semidefinite matrix
$S \in \bbR^{m \times m}$ such that $p(x_1,\dots,x_n) = [x]_r^\top S [x]_r$,
where we denote by $[x]_r$ the vector 
of length $m:=\abs{\bbN^n_{\leq r}}=\binom{r+n}{n}$
containing all monomials in $\bbR[x_1,\dots,x_n]$
up to degree $r$ in lexicographic order.
To see this, start with a sum-of-squares polynomial $p=p_1^2+\dots+p_k^2$.
Denote by $\ell_1,\dots,\ell_k$ the coefficient vectors
of $p_1,\dots,p_k$, i.e. $p_i = \ell_i^\top [x]_r$. Using this notation 
we can write
\begin{align*}
    p = (\ell_1^\top [x]_r)^\top \ell_1^\top [x]_r + \dots + (\ell_k^\top [x]_r)^\top \ell_k^\top [x]_r
= \sum_{i=1}^k [x]_r^\top (\ell_1 \ell_1^\top +\dots+\ell_k \ell_k^\top) [x]_r.
\end{align*}
Now $S:=\ell_1 \ell_1^\top +\dots+\ell_k \ell_k^\top$ is a positive semidefinite $m \times m$ matrix.
On the other hand, every positive semidefinite matrix $R$
can be decomposed as $R=L^\top L$, e.g. by Cholesky decomposition,
where $L$ is a real $m \times m$ matrix whose rows we denote by $\ell_1,\dots,\ell_m$,
so if $p = [x]_m^\top R [x]_m$ we again obtain the sum-of-squares decomposition
$p = \sum_{i=1}^m (\ell_i^\top [x]_r)^2$.

Therefore given a bound on the degrees of the polynomials involved,
we can find a weighted sum-of-squares decomposition 
by semidefinite programming.
To be reasonably fast this is done using floating point arithmetic.
To obtain an exact rational weighted sum-of-squares decomposition
we proceed as follows, following \cite{peyrlComputingSumSquares2008a}.
Instead of relying on arguments on the position
of our Gram matrix in the cone of positive semidefinite matrices,
we will use the fact that we know bounds on the
absolute value of our variables, see
\eqref{eq:bound-projection-entries}. For more elaborate
rounding procedures for weighted sum-of-square decompositions,
see \cite{magronExactReznickHilbertArtin2021,
  davisRationalDualCertificates2024, davisDualCertificatesEfficient2022}
Assume we have
\begin{align*}
  -1 \approx [x]_r^\top \tilde{Q} [x]_r + \sum_{j=1}^t \tilde{q}_j g_j
  + \sum_{i=1}^s \tilde{p}_i f_i 
\end{align*}
with floating point polynomials 
$\tilde{p}_i \in \bbR[x]$, $\tilde{q}_j \in \bbR[x]^2$ and a positive semidefinite floating point matrix $\tilde{Q}$.

We first approximate the polynomials $\tilde{p}_i, \tilde{q}_j$ by
polynomials $p_i \in \bbQ[x]$ and $q_j \in \bbQ[x]^2$.
Similarly we approximate $\tilde{Q}$ by
a rational matrix $Q'$.
Then we project $Q'$ orthogonally to the space
\begin{align*}
  \{T \in \bbR^{d \times d} \setsep T \text{ is symmetric }, -1
  -\sum_{j=1}^t q_j g_j
  - \sum_{i=1}^s p_i f_i  = [x]_r^\top T [x]_r\}.
\end{align*}
As shown in \cite[Proposition 7]{peyrlComputingSumSquares2008a},
this projection $Q$ of $Q'$
is explicitly given by
\begin{align*}
    Q_{\alpha,\beta} = Q'_{\alpha,\beta} - \frac{1}{\#(\alpha+\beta)} (\sum_{\alpha'+\beta'=\alpha+\beta} Q'_{\alpha',\beta'} - h_{\alpha+\beta}).
\end{align*}
where $\alpha, \beta,\alpha',\beta' \in \bbN_{\leq r}^n,$
\[\#(\gamma) := \{(\alpha',\beta') \in \bbN^n_{\leq r} \times \bbN^n_{\leq r} \setsep \gamma=\alpha'+\beta'\}\]
and
\[h(x) := \sum_{\gamma} h_\gamma x^\gamma := -1 - \sum_{j=1}^t q_j(x) g_j(x)
  - \sum_{i=1}^s p_i(x) f_i(x).\]

Now we use the fact, that we know
a bound $\Omega$ for the absolute value of our variables
and that $\Omega^2 - x_i^2 \in \calG$ for $i \in \{1,\dots,n\}$, see
the discussion in the beginning
of this section.
As shown in \Cref{thm:offset} there is $\Delta \in \bbQ_{>0}$
such that
$\Delta - [x]_r^\top [x]_r  \in M_\bbQ(\calG)$. 
Notice that $Q$ is a symmetric rational matrix.
but is not necessarily positive semidefinite. If we choose
our rational approximations good enough,
it will nevertheless be sufficiently close to $\tilde{Q}$ such that
the eigenvalues of $Q$ will be larger than $-\frac{1}{4\Delta}$.

Then $Q + \frac{1}{2\Delta}I$ is positive definite and we
can compute its $LDL^\top$-decomposition,i.e. rational square matrices $L,D$ with $D$ diagonal and positive such that
\begin{align*}
    Q+ \frac{1}{2\Delta} I = L^\top DL.
\end{align*}
Denote the rows of $L$ by $\ell_i, i=1,\dots,m=\binom{r+n}{n}$.
Set
\begin{align*}
  q_0(x):=[x]_r^\top (Q + \frac{1}{2\Delta}) [x]_r &:= [x]_r^\top L^\top D L [x]_r \\
  &= \sum_{i=1}^{m} D_{ii} (\ell_i^\top [x]_r)^2
    \in \bbQ[x]^2.
\end{align*}
By the definition of $Q$ we have
\begin{align*}
  -1  - \sum_{j=1}^t g_j q_j - \sum_{i=1}^s f_i p_i
  &= [x]_r^\top Q [x]_r
    = [x]_r^\top (Q+ \frac{1}{2\Delta} I) [x]_r - \frac{1}{2\Delta} [x]_r^\top [x]_r\\
  &= q_0(x) - \frac{1}{2} +  \frac{1}{2\Delta}(\Delta - [x]_r^\top [x]_r) \\
  -1+\frac{1}{2}
  &=q_0(x)+ \sum_{j=1}^t g_j(x) q_j(x)  + \sum_{i=1}^s f_i(x) p_i(x)  +
    \frac{1}{2\Delta}(\Delta - [x]_r^\top [x]_r)\\
  -1
  &= 2 (q_0(x)+\sum_{j=1}^t g_j(x) q_j(x)+\sum_{i=1}^s f_i(x) p_i(x))
    + \frac{1}{\Delta}(\Delta - [x]_r^\top [x]_r))
\end{align*}
We have already seen that $\Delta - [x]_r^\top [x]_r \in M_\bbQ(G )$,
so the whole right hand side lies in $M_\bbQ(\calG) + I_\bbQ(\calF)$.
We thus produced a rational weighted sum-of-squares decomposition of $-1$ over
our basic semi-algebraic set $\calM$. This constitutes
a certificate that $\calM$ is empty and hence a proof of
\Cref{thm:lower-bound}. For
the concrete implementation details and the obtained certificate see \Cref{sec:impl-sos}.

We finish this section by sketching how the previously described
methods can alternatively be applied to \Cref{op:kernel-formulation-compact}.

In order to show that $s_k(\calC) \geq \alpha^*$, we have
to show that
\begin{align*}
  \calN_{\calC}(\alpha^*) := \{u,w \in \bbR^d \setsep &w^\top u \geq 0, \\
                               &u, w \in [-1,1]^d, \\
                               &\exists i \in \{1,\dots,d\}: u_i \in \{-1,1\},\\
                               &\exists j\in \{1,\dots,d\}: w_j= 1,\\
                               & h^\top v w^\top u-w^\top v h^\top u -
                             \alpha^* w^\top u\leq  0,\quad
                             \text{for }v \in V', h \in H \}= \emptyset.
\end{align*}
Notice however that $\calN_\calC(\alpha^*)$ is not 
a basic semialgebraic set so we can not apply Putinar's
Positivstellensatz
directly.
We can however represent $\calN_\calC(\alpha^*)$
as the union of $2d^2$ basic semialgebraic sets
of the form $\calN_{\calC,i,j,\tau}$
with $i,j \in \{1,\dots,d\}$ and $\tau \in \{-1,1\}$,
where
\begin{align*}
  \calN_{\calC,i,j,\tau}(\alpha^*) := \{u,w \in \bbR^d \setsep &w^\top u \geq 0, \\
                               &u_k^2 \leq 1,\, w_k^2 \leq 1, \quad
                                 \text{ for }k \in \{1,\dots,d\}, \\
                               & u_i = \tau, \\
                               & w_j = 1,\\
                               & h^\top v w^\top u-w^\top v h^\top u -
                               \alpha^* w^\top u\leq  0,\quad
                               \text{for } v \in V', h \in H \}.
\end{align*}
Since we fix two of the entries of $v$ and $w$, this
is effectively a basic semialgebraic set described by polynomials
in $2d-2$ variables, which is a significant reduction compared
to the $d^2-1$ variables we had before. In practice
this advantage seems to be outweighed by
needing significantly higher degree for the polynomials
in the weighted sum-of-squares decomposition to achieve
comparable lower bounds.

We can therefore obtain a certificate that $\calC$ is $\alpha^*$-shady,
if we produce a rational weighted sum of squares decomposition
for each of the set $\calN_{\calC,i,j,\tau}(\alpha^*),\; i,j \in \{1,\dots,d\}, \tau \in \{-1,1\}$.

\section{A Proof Certificate Based on Linear Duality}
\label{sec:eps-dense-farkas}

In this section we provide another way to produce a certificate for
a lower bound $\alpha^*$ for
\Cref{op:kernel-formulation}. This will produce
an alternative computer assisted proof of \Cref{thm:lower-bound}.
Again we want to show that the
constraints become infeasible for $\alpha \leq \alpha^*$.  For a
fixed vector $w \in \bbR^d$, we can accomplish this using duality as follows.
Let $A_{w,\alpha^*}$ be the $\abs{V'}\abs{H} \times d$ matrix
whose rows are given by
$h^\top v w^\top - w^\top v h^\top - \alpha^* w^\top$ for
$v \in V', h \in H$.  We can rewrite the infeasibility of our problem as
\begin{align*}
  \{ u \in \bbR^d \setsep (-w^\top) u < 0, \;A_{w,\alpha^*} u \leq 0\} = \emptyset.
\end{align*}
Farkas's lemma (see e.g. \cite[Section
5.8]{boydConvexOptimization2004}) states that this is equivalent to
\begin{align*}
  \calD_{w,\alpha^*}:=\{y \in \bbR^d \setsep A_{w,\alpha^*}^\top y = w, \quad y \geq 0\} \neq \emptyset.
\end{align*}
So to show that all projections with range $w^\bot$ have norm
larger than $\alpha^*$, we have
to show that $w$ is contained in the convex cone generated by the rows
of $A_{w,\alpha^*}$. Since the entries of that matrix are rational,
we can also find positive rational solutions $y_w$.
This vector $y_w \in \calD_{w,\alpha^*} \cap \bbQ^{\abs{V'}\abs{H}}$ serves as a certificate that $w^\bot$
is never the image of a projection with norm less or equal to $\alpha^*$.
In terms of the relative projection constant this means that
\begin{align*}
  \lambda(w^\perp, (\bbR^d,\norm{\cdot}_\calC)) \geq \alpha^*.
\end{align*}

Now we can of course not produce certificates for all $w \in \bbR^d$,
so we only produce certificates for a
sufficiently dense set $\tilde{W}$.
This also covers all vectors in \[W := \{\gamma w \setsep w \in \tilde{W},
\gamma \in \bbR \setminus \{0\}\}.\]
After we obtained the certificates for all vectors in $\tilde{W}$,
we know that a projection
with image $w^\perp, w \in W$ has norm at least $\alpha^*$.
For all other projections $P$ we
get a lower bound for $\norm{P}_\calC$ slightly worse than $\alpha^*$
as follows.
Assume that $W$ contains an $\eps$-dense subset
of the Euclidean unit sphere, i.e. for all $x \in \bbR^d$
with $\norm{x}_2=1$ there is $w \in W$ with $\norm{w-x}_2 \leq \eps$.
For all $w' \not\in W$ we use the following estimates:
Let $\calC$ be the constant from \Cref{lem:compare-to-euclidean}, so we have $\norm{A}_\calC \leq C_\calC
\norm{A}_2$ for all $A \in \bbR^{d \times d}$.
Let $P$ be a projection whose image is perpendicular
to $w' \in \bbR^d$ with $\norm{w'}_2=1$. There is $w \in W$
with $\norm{w-w'}_2 \leq \eps$.
Therefore there is a rotation\footnote{
  More precisely $Q$ is an orthogonal map which is a rotation
  on the plane spanned by $w$ and $w'$ and which fixes the
  orthogonal complement of this plane.}
  $Q$
mapping $w'$ to $w$ such that $\norm{Q-I}_2 \leq \eps$
as well as $\norm{Q^{-1}-I}_2 \leq \eps$. Observe that $\norm{Q}_2=1$
and that $QPQ^{-1}$ is a projection
with image $w^\perp$.
Now
\begin{align*}
  \norm{P}_\calC
          &\geq \norm{(QPQ^{-1}}_\calC-\norm{P-QPQ^{-1}}_\calC \\
  \norm{P-QPQ^{-1}}_\calC
          &\leq \norm{P-QP}_\calC + \norm{QP-QPQ^{-1}}_\calC \\
          &\leq \norm{P}_\calC \norm{I-Q}_\calC + \norm{QP}_\calC \norm{I-Q^{-1}}_\calC \\
          &\leq C_\calC\norm{P}_\calC \norm{I-Q}_2
            + C_\calC^2\norm{Q}_2\norm{I-Q^{-1}}_2\norm{P}_\calC \\
          &\leq \eps(C_\calC+C_\calC^2) \norm{P}_\calC \\
  \norm{P}_\calC
          &\geq
            \frac{\norm{QPQ^{-1}}_\calC}{1+\eps(C_\calC+C_\calC^2)}
            \geq \frac{\alpha^*}{1+\eps(C_\calC+C_\calC^2)}.
\end{align*}

For our concrete polytope $\calI$ from
\Cref{thm:lower-bound} we get
\begin{align}
  C_\calI \approx \sqrt{\frac{132}{27}} \approx 2.253, \label{eq:spectral-norm-bound-I}
\end{align}
which is too large for our purposes.
Recall, however, that according to \Cref{lem:transformation-shadiness}
we can as well obtain lower
bounds for $s_{2}(T\calI)=s_2(\calI)$ for an appropriately chosen
invertible matrix $T$.
For $\calJ:=T \calI$ with
\begin{align*}
\begingroup
\renewcommand*{\arraystretch}{1.2}
T:=\begin{pmatrix}
\frac{11}{10} & \frac{11}{10} & \frac{11}{10} \\
\frac{7}{10} & \frac{1}{5} & \frac{-9}{10} \\
\frac{3}{5} & \frac{-9}{10} & \frac{3}{10} \\
\end{pmatrix}
\endgroup
\end{align*}
we get
\begin{align}
  C_\calJ \leq 1.55 \label{eq:spectral-norm-bound-II}
\end{align}
see \Cref{sec:impl-general}. This matrix $T$ approximately brings
$\calI$ into John position, i.e. a Euclidean ball is (almost) the minimal
inscribed ellipsoid of $\calJ$.

Now set $\tilde{W} = \bigcup_{k,\ell \in \{-n,\dots,n\}}
\{(1,\frac{k}{n},\frac{\ell}{n}),(\frac{k}{n},1,\frac{\ell}{n}),(\frac{k}{n},\frac{\ell}{n},1)\}$.
Let $v \in \bbR^3$ be in the Euclidean unit sphere.
By permuting coordinates if necessary, assume that $\abs{v_1} \geq \abs{v_2} \geq \abs{v_3}$.
Then there is $\tilde{w} \in \tilde{W}$
such that \[\tilde{w}_1=1,\; \abs{\frac{v_2}{v_1}-\tilde{w}_2}\leq \frac{1}{2n},\; \abs{\frac{v_3}{v_1}-\tilde{w}_3} \leq
\frac{1}{2n}.\] 
Set $\alpha:= \frac{1}{v_1}$,
$\beta:= \frac{\abs{v_1} \norm{\tilde{w}}_2}{v_1}$,
$\tilde{v}:= \alpha v$
and $w:=\frac{1}{\beta} \tilde{w}$.Then
$\norm{w}_2 = \norm{v}_2=1$.
Notice that $\alpha$ and $\beta$ have the same sign and are both
larger or equal one in absolute value, so in particular $\alpha\beta
\geq 1$. Hence
\begin{align*}
  \norm{\tilde{v}-\tilde{w}}_2^2 - \norm{v-w}_2^2 
  &= \norm{\alpha v -\beta w}_2^2 -\norm{v-w}_2^2 \\
  &=(\alpha^2-1)\norm{v}^2+(\beta^2-1) \norm{w}^2 + 2(1-\alpha\beta) v^\top w\\
  &\geq \alpha^2+\beta^2-2 + 2 - 2\alpha\beta = (\alpha-\beta)^2 \geq 0
\end{align*}
Since $\tilde{v}_1=\tilde{w}_1=1$ we get
$\norm{v-w}_2 \leq \norm{\tilde{v}-\tilde{w}}_2 \leq
 \sqrt{\frac{2}{4n^2}} = \frac{1}{n\sqrt{2}}$.
 This shows that $W$ is $\frac{1}{n\sqrt{2}}$-dense in the Euclidean
 unit sphere and thus we can set $\eps=\frac{1}{n\sqrt{2}}$ in the calculation above.
  
If we now further set $n:=1400$ and $\alpha^*:=\frac{84}{83}$
we get
\begin{align}
  \norm{P}_\calC \geq
  \frac{\frac{84}{83}}{1+\frac{1.55+1.55^2}{1400\sqrt{2}}} \geq 1.01.
  \label{eq:final-norm-estimate}
\end{align}

We will discuss in \Cref{sec:impl-farkas} how we obtained
the certificates for the vectors in $\tilde{W}$.
Together with these certificates,
\eqref{eq:final-norm-estimate} shows that $\calJ$, and hence by
\Cref{lem:transformation-shadiness}
also $\calI$, have a shadiness constant of at least $1.01$. 
This again proves \Cref{thm:lower-bound}.

\section{Non-Shadiness of Polytopes with Few Vertices}
\label{sec:non-shadiness}
In this section we show
that no three-dimensional centrally symmetric convex polytope with ten or fewer vertices is
shady.

This is more or less a direct consequence of the following lemma, for
which
we need the notion of centrally symmetric
triangulated polyhedron.
By this we mean a 
centrally symmetric convex polytope $\calC$ with zero-dimensional faces $V$
and one-dimensional faces $E$ together with sets $V' \supseteq V$
and $E' \supseteq E$, called the \emph{vertices} and \emph{edges} of the triangulation,
such that:
\begin{enumerate}[(a)]
\item $\tilde{V}$ consists of points on the boundary of $\calC$,
\item $\tilde{E}$
consists of line segments with pairwise disjoint interiors, contained in the boundary of $\calC$
and connecting vertices in $V'$,
\item $\tilde{V}=-\tilde{V}$ ,
  $\tilde{E}=-\tilde{E}$,
  \item for every facet $G$ of $\calC$, the vertices and edges contained
in $G$ form a triangulation of $G$, i.e. partition
$G$ into triangles. We denote the entirety of these triangles on
all facets by $\tilde{F}$.
\end{enumerate}
The triangulation naturally defines an undirected graph
with vertex set $\tilde{V}$ and edge set $\tilde{E}$.
By an \emph{$n$-cycle in the triangulation} we mean an $n$-cycle in this
graph, i.e. pairwise distinct vertices $v_1,\dots,v_n$ such that
$v_i$ and $v_{i+1}$ for $i \in \{1,\dots,n-1\}$, as well as $v_n$
and $v_1$ are connected by an edge.

\begin{lemma}
  \label{lem:10-vertices-combinatorial}
  Let $\calC \subseteq \bbR^3$ be a centrally symmetric triangulated
  convex polytope with $10$ vertices (in the triangulation). Then there is a pair of vertices $v$, $w$
  such that $v,w,-v,-w$ form a $4$-cycle.
\end{lemma}
\begin{proof}
  Let $\tilde{V},\tilde{E},\tilde{F}$ be the sets of vertices, edges and triangles of the
  triangulation. By Euler's polyhedron formula
  we have
  \begin{align*}
    \abs{\tilde{V}}-\abs{\tilde{E}}+\abs{\tilde{F}}=2.
  \end{align*}
  Since every triangle has three sides and
  every edge of the triangulation is contained in exactly two triangles, we have
  $2\abs{\tilde{E}}=3\abs{\tilde{F}}$. Combining these equations gives
  us $\abs{\tilde{V}}-\frac{1}{3}\abs{\tilde{E}}=2$,
  i.e. $30=3\abs{\tilde{V}}=\abs{\tilde{E}}+6$
  and so $\abs{\tilde{E}}=24$.
  If all vertices have degree at most $4$, then we would have
  $40=4\abs{\tilde{V}} \geq 2\abs{\tilde{E}}=48$.
  Hence there is a vertex $v$ of degree at least $5$.
  The vertex $v$ can not be connected by an edge to $-v$ in the triangulation,
  since that edge would pass through the origin, which cannot lie on the boundary of $\calC$
  There are four remaining equivalence classes of vertices,
  if we identify vertices which are opposite each other.
  By the pigeonhole principle, there must be a vertex $w \neq v$ such
  that both $w$ and $-w$ are neighbors of $v$.
  Since $\calC$ is invariant under $x \mapsto -x$,
  $w$ and $-w$ are also connected to $-v$,
  hence $v,w,-v,-w$ form a $4$-cycle.
\end{proof}
\begin{figure}
  \begin{center}
    \includegraphics[scale=0.8]{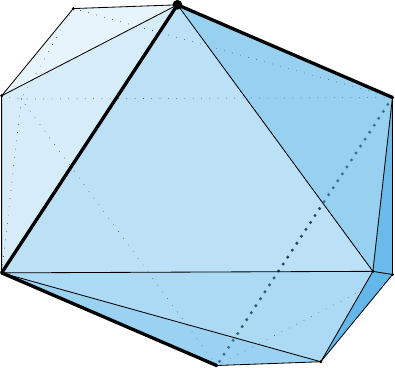}
    \caption{A centrally symmetric triangulated convex polytope. A vertex with degree 5 and a
      corresponding
      4-cycle are highlighted}
    \label{fig:triang}
    \end{center}
\end{figure}

With this preparation it is quite easy to prove:
\tenvertices*
\begin{proof}
  By adding edges and vertices on the faces of $\calC$, we triangulate
  $\calC$ such that the triangulation is centrally symmetric.
  By \Cref{lem:10-vertices-combinatorial} there are vertices $v$,$w$
  such that 
  $v,w,-v,-w$ form a $4$-cycle in the triangulation, see
  \Cref{fig:triang}.
  In particular these four vertices form a plane quadrilateral
  whose edges lie on the boundary of $\calC$.
  Let $U$ be the subspace spanned by $v$ and $w$.
  $\calC$ intersects $U$ precisely
  in the quadrilateral defined by $v,w,-v,-w$.
  Let $F$ and $G$ be supporting hyperplanes of $\calC$ containing
  the edges $(v,w)$ and $(v,-w)$ respectively.
  Let $F'$ and $G'$ be the linear subspaces
  parallel to $F$ and $G$, respectively.
  Set $H:=F' \cap G'$. Let $P$ be the unique projection
  with image $U$ and kernel $H$.
  Then the image of $\calC$ under $P$ equals
  the quadrilateral $\calC \cap U$, which is contained in $\calC$,
  hence $\norm{P}_\calC=1$. 
\end{proof}

\section*{Acknowledgments}
I would like to thank Hella Epperlein for pointing out the degree argument that
significantly shortened the proof of \Cref{thm:10-vertices-geometric}.

\appendix
\crefalias{section}{appendix}
\section{General Implementation Details}
\label{sec:impl-general}
We implemented our algorithms to produce the certificates described
in \Cref{sec:sum-of-squares} and \Cref{sec:eps-dense-farkas} in
\name{Julia}. The code can be found at
\url{https://github.com/JeremiasE/ShadyPolytopes.jl}.
All references to files refer to this repository.

The following packages were essential:
\begin{enumerate}
\item
  \texttt{SumOfSquares.jl} -- Sum-of-squares decomposition
  \cite{legat2017sos,weisser2019polynomial}.
\item \texttt{Clarabel.jl} -- interior point solver for convex
  optimization problems \cite{goulartClarabelInteriorpointSolver2024},
  used for solving the semidefinite optimization problems
  produced by \texttt{SumOfSquares.jl}
  \item
  \texttt{HiGHs.jl} -- wrapper around the linear solver HiGHs
  \cite{huangfuParallelizingDualRevised2018}.
  \item \texttt{Singular.jl} -- wrapper around the computer algebra system
  Singular used for Gröbner basis computations
  \cite{wolfram2024Singular}.
\item\texttt{Polyhedra.jl} -- For general polyhedral computations
  \cite{legat2023polyhedral}.
\end{enumerate}

The computation of the constants
$C_\calI$ and $C_\calJ$ is done in the \texttt{Jupyter} notebook
\begin{lstlisting}
  notebooks/Analyze_perturbed_icosahedron.ipynb
\end{lstlisting}
The simple test for shadiness of $\calI$ in
\Cref{ex:simple-shadiness-test}
is performed in the \texttt{Jupyter} notebook
\begin{lstlisting}
  notebooks/Simple_shadiness_test.ipynb
\end{lstlisting}

The code for all the computations is contained
in the \texttt{src} subfolder.
The commands to produce the certificates are contained in the 
\texttt{script} subfolder and the precomputed certificates
are located in the \texttt{data} subfolder.

All computations were done on an Intel Core i7-4770K CPU (3.50GHz, 4 cores)
running Fedora 41 and Julia 1.11.4.

\section{Implementation Details for the Sums-of-Squares Approach}
\label{sec:impl-sos}

In this section we give an overview of the implementation
generating the sums-of-squares certificate discussed
in \Cref{sec:sum-of-squares}
using the notation from that section.
The certificate is calculated by calling
\begin{lstlisting}
  julia scripts/generate_sos_certificate.jl
\end{lstlisting}
in the \texttt{scripts} subfolder.
This produces a Julia file \texttt{sos\_certificate.jl} in the
\texttt{data} subfolder, which contains the weighted
sum-of-squares decomposition, as follows:
We start by obtaining a floating-point weighted sum-of-squares
decomposition using the \texttt{SumOfSquares.jl} package.
As the SDP-solver we use \texttt{Clarabel}.
\texttt{SumOfSquares.jl} takes the equality constraints $f_i(x)=0, i =\{1,\dots,s\}$ into account
by computing a Gröbner basis $f'_1,\dots,f'_{s'}$ and reducing
all other constraints appropriately.
This results in $125595$ scalar variables in the semidefinite optimization problem.
As the upper bound for the squares of entries of $P$, see
\eqref{eq:bound-projection-entries}
and \eqref{eq:spectral-norm-bound-I},
we get $\Omega = 1.01 \cdot C_\calC = 1.01 \cdot \sqrt{\frac{137}{27}}
\approx 2.23$.
We use \texttt{SumOfSquares.jl}
to find a weighted sum-of-squares decomposition of the
constant polynomial $-1$
where the maximal degree of polynomials is $5$.
For our inequality constraints defined by linear polynomials $g_j$,
this leads to sums of squares of quadratic polynomials
in the coefficient polynomials $\tilde{q}_j$.
Similarly, for our quadratic inequality constraints $g_j$
as well as our quadratic equality constraints $f_i$
we only get sums of squares of linear polynomials in the coefficients
$\tilde{q}_j$ and $\tilde{p}_i$.

This leads to $r=3$ in the calculation of the projected Gram matrix
and a value of $\Delta = \frac{2894536936604222153}{164025000000000}$.

The polynomials $\tilde{q}_j, j \in \{1,\dots,t\}$
as well as the Gram matrix $\tilde{Q}$ can be obtained directly
via the methods implemented in \texttt{SumOfSquares.jl}.
Due to the Gröbner basis technique mentioned above, the floating point
polynomials $\tilde{p}_j, j \in \{1,\dots,s\}$ and
their rounded version $p_j, j \in \{1,\dots,s\}$ are not directly accessible.
To obtain them, we first compute
\[h := -1 - [x]_r^\top \tilde{Q} [x]_r - \sum_{j=1}^t \tilde{q}_j
  g_j.\]
By using the Euclidean algorithm and the Gröbner basis $f_1',\dots,f'_{s'}$
we can obtain floating point polynomials
\begin{align*}
  \tilde{p}'_1,\dots, \tilde{p}'_{s'}
\end{align*}
such that
\begin{align*}
  h = \sum_{i=1}^{s'} \tilde{p}'_i f'_i
\end{align*}

Again we round these polynomials to $p'_1,\dots,p'_{s'} \in \bbQ[x]$.
We use \texttt{Singular.jl} to compute polynomials $e_{i,j} \in \bbQ[x],\; i=1,\dots,s',\;
j=1,\dots,s$
\begin{align*}
  f_i' = \sum_{j=1}^{s} e_{i,j} f_{j}
\end{align*}
Finally we obtain the polynomials $p_j \in \bbQ[x]$
as
\begin{align*}
  p_j:= \sum_{i=1}^{s'} e_{i,j} p'_i
\end{align*}
Next we project the computed Gram matrix $\tilde{Q}$
as described in \Cref{sec:sum-of-squares}
and obtain a rational matrix $Q$.

We use \texttt{Julia}'s \texttt{bunchkaufman} method to compute
a rational LDL-decomposition of $Q + \frac{1}{2\Delta}$
to obtain the rational matrices $L$ and $D$ satisfying
$Q + \frac{1}{2\Delta} = LDL^\top$
where $D$ is diagonal.

The proof certificate is verified by calling:
\begin{lstlisting}
  julia scripts/check_sos_certificate.jl
\end{lstlisting}

It takes roughly 20 minutes to generate the certificate 
and 100 seconds to check it.

\section{Implementation Details for the Linear Duality Approach}
\label{sec:impl-farkas}
The proof certificate for \Cref{thm:lower-bound}
using the approach described in \Cref{sec:eps-dense-farkas}
can be generated by calling

\begin{lstlisting}
  julia --threads 3 scripts/generate_farkas_certificate_john_position.jl
\end{lstlisting}
  
The output consists of three files in the \texttt{data} subfolder
\begin{align*}
  \text{\texttt{farkas-certificates-opt-ico-jp-1400-84\_83-i.csv.gz}},\,i=1,\dots,3
\end{align*}
corresponding to the facet $F_i=\{x \in [-1,1]^3 \setsep x_i=1\}$ of
$B_\infty$,
$\alpha^*=\frac{84}{83}>1.01$ and $n=1400$.
These are CSV files compressed with GZIP; when uncompressed, each
file has a size of about 1.1GiB.

Each line in the $i$-th file corresponds to a point $w \in F_i$
and has the form
\[(w_1,w_2,w_3,k_1,k_2,k_3,y_{k_1},y_{k_2},y_{k_3}),\]
e.g.
\begin{lstlisting}
  1//1;-39//40;-164//175;40;57;115;
  43084159464618720881//5777554117512961187;
  14135303314411071435//5777554117512961187;
  3689530486357540849//5777554117512961187
\end{lstlisting}
This corresponds to
$w=(1,\frac{-39}{40},-\frac{164}{175})$
and $y \in \bbR^{140}$.
\begin{align*}
  y_i=\begin{cases}
    \frac{43084159464618720881}{5777554117512961187} &\text{if } i=40 \\
    \frac{14135303314411071435}{5777554117512961187} &\text{if } i=57 \\
    \frac{3689530486357540849}{5777554117512961187} &\text{if }  i=115\\
    0 &\text{otherwise}
  \end{cases}
\end{align*}
This $y$ satisfies
$A_{w,\alpha^*}^\top y=w$, where
$A_{w,\alpha^*}$
is the matrix with rows $h^\top v w^\top - w^\top v h^\top - \alpha^* w^\top$ for
$v \in V', h \in H$.

The vectors $y$ are calculated as follows.
We use \texttt{HiGHS} to find a non-negative solution $y$ of
$A_{w,\alpha^*}^\top y = w$. This will only be a floating point solution.
To obtain a rational solution, we first observe
that, by Carathéodory's theorem,
$w$ is in the convex cone generated by the rows of $A_{w,\alpha^*}$
iff $w$ is the non-negative linear combination of $d$
rows of $A_{w,\alpha^*}$. In other words, if there is a non-negative
solution $y$
to $A_{w,\alpha^*}^\top y = w$, then there also is a
non-negative solution in which
at most $d$ entries are non-zero. It turns out that
\texttt{HiGHS} actually returns these sparse solutions.
We can then also find a non-negative rational solution
with the same support
by simply solving the corresponding
$3 \times 3$ rational linear equation system.

The certificate can be checked with the following call:
\begin{lstlisting}
  julia --threads 3 scripts/check_farkas_certificate_john_position.jl
\end{lstlisting}
This script checks for each line in the three certificate
files that the corresponding $y$ is positive
and solves $A_{w, \alpha^*}^\top y = w$.

It takes approximately 20 hours to generate the proof certificate
and 10 hours to check it. This approach is significantly slower
than the sum-of-squares approach, but has the advantage that
both the generation as well as the validation of the certificate
is trivially parallelizable. We use this fact
to generate and validate the three certificate files
in parallel.
\printbibliography

\end{document}